\documentclass[12pt, reqno]{amsart}
\makeatletter
\@namedef{subjclassname@1991}{$\mathrm{1991}$ Mathematics Subject Classification}
\@namedef{subjclassname@2000}{$\mathrm{2000}$ Mathematics Subject Classification}
\@namedef{subjclassname@2010}{$\mathrm{2010}$ Mathematics Subject Classification}
\@namedef{subjclassname@2020}{$\mathrm{2020}$ Mathematics Subject Classification}
\makeatother
\usepackage{amsmath,amsthm, amscd, amsfonts, amssymb, graphicx, color}
\usepackage[bookmarksnumbered, colorlinks, plainpages,linkcolor=blue,urlcolor=blue,citecolor=blue]{hyperref}
\newtheorem{thm}{Theorem}[section]
\newtheorem{cor}[thm]{Corollary}
\newtheorem{lem}[thm]{Lemma}

\newtheorem{defn}[thm]{Definition}

\newtheorem{exam}[thm]{Example}
\numberwithin{equation}{section}

\begin{document}

\title{generalized EP properties of $(b,c)$-inverses}

\author{Huanyin Chen}
\address{School of Big Data, Fuzhou University of International Studies and Trade, Fuzhou 350202, China}
\email{<huanyinchenfz@163.com>}

\subjclass[2020]{16U50, 15A09, 46H05.} \keywords{weak group inverse; $(b,c)$-inverse; $(b,c)$-EP inverse; generalized group inverse; Banach algebra.}

\begin{abstract}In this paper, we introduce the notion of the generalized $(b,c)$-EP inverse within the framework of a *-Banach algebra. This concept emerges as a logical extension of the weak group inverse and EP-like property, which is applicable to complex matrices and bounded linear operators in Hilbert spaces. We provide its characterizations in relation to its associated decomposition and the generalized Drazin inverse. A polar-like property for the generalized $(b,c)$-EP inverse is presented. Furthermore, we reveal an intrinsic connection between the generalized $(b,b)$-EP inverse and the converse law for the generalized group inverse.\end{abstract}

\maketitle

\section{Introduction}

A Banach algebra is called a Banach *-algebra if there exists an involution $*: x\to x^*$ satisfying $(x+y)^*=x^*+y^*, (\lambda x)^*=\overline{\lambda} x^*, (xy)^*=y^*x^*, (x^*)^*=x$. An element $a\in \mathcal{A}$ has group inverse provided that there exists $x\in \mathcal{A}$ such that $$ax^2=x, ax=xa, xa^2=a.$$ Such $x$ is unique if exists, denoted by $a^{\#}$, and called the group inverse of $a$. Evidently, a square complex matrix $A$ has group inverse if and only if $rank(A)=rank(A^2)$ (see~\cite{B,C3,M}).

Let $\mathcal{A}$ be a Banach *-algebra. The involution $*$ is proper if $x^*x=0\Longrightarrow x=0$ for any $x\in \mathcal{A}$.
Let ${\Bbb C}^{n\times n}$ be the Banach algebra of all $n\times n$ complex matrices, with conjugate transpose $*$ as the involution.
Then the involution $*$ is proper.

Let $\mathcal{A}$ be a Banach *-algebra with a proper involution $*$. An element $a\in \mathcal{A}$ has weak group inverse if there exist $x\in \mathcal{A}$ and $n\in \Bbb{N}$ such that $$ax^2=x, (a^*a^2x)^*=a^*a^2x, xa^{n+1}=a^n.$$ If such $x$ exists, it is unique, and denote it by $a^{\tiny\textcircled{W}}$ (see~\cite{C2,D1,MD5,W1,Z1,Z2}). Mosic and Zhang introduced and studied weak group inverse for a Hilbert space operator $A$ in $\mathcal{B}(X)^d$ (see~\cite{MD4}).

Recently, Chen and Sheibani introduced and studied a new generalized inverse as a generalization of group and weak group inverse. An element $a$ in a Banach *-algebra has generalized group inverse if there exists $x\in \mathcal{A}$ such that $$x=ax^2, (a^*a^2x)^*=a^*a^2x, \lim\limits_{n\to \infty}||a^n-xa^{n+1}||^{\frac{1}{n}}=0.$$ Such $x$ is unique if it exists and is called the generalized group inverse of $a$. We denote it by $a^{\tiny\textcircled{g}}$. Many properties of generalized group inverse were presented in ~\cite{C1,C7,C4}. Evidently, $\{ ~\mbox{group inverse}\}\subsetneq \{ ~\mbox{weak group inverse}\}\subsetneq \{ ~\mbox{generalized group inverse}\}$.

Let $a,b,c\in R$. An element $a$ has $(b,c)$-inverse provide that there exists $x\in R$ such that $$xab=b, cax=c ~\mbox{and}~ x\in bRx\bigcap xRc.$$ If such $x$ exists, it is unique and denote it by $a^{(b,c)}$. It is well known that $x=a^{(b,c)}~\mbox{if and only if} ~x=xax, xR=bR,~\mbox{and} ~Rx=Rc.$ We refer the reader to ~\cite{BK,DM,K1,K,S} for various properties of $(b,c)$-inverse.

Combing group inverse and $(b,c)$-inverse, Wu and Chen introduced and studied EP-like properties for the $(b,c)$-inverse (see~\cite{WC}).

\begin{defn} An element $a\in \mathcal{A}$ has $(b,c)$-EP inverse if $a\in \mathcal{A}^{\#}$ and $a^{\#}=a^{(b,c)}$. We use $\mathcal{A}_{b,c}^{\#}$ to denote the set of all $(b,c)$-EP invertible element in $\mathcal{A}$.\end{defn}

Many properties of $(b,c)$-EP inverse are established in ~\cite{WC}. Various related results can be found in ~\cite{D1,D2}.

The objective of this paper is to examine when the generalized group inverse of a Banach element can be represented as its $(b,c)$-inverse. We adopt

\begin{defn} An element $a\in \mathcal{A}$ has generalized $(b,c)$-EP inverse if $a\in \mathcal{A}^{\tiny\textcircled{g}}$ and $a^{\tiny\textcircled{g}}=a^{(b,c)}$. We use $a_{b,c}^{\tiny\textcircled{g}}$ to stands for $a^{\tiny\textcircled{g}}$. The set of all generalized $(b,c)$-EP invertible elements in $\mathcal{A}$ is denoted by $\mathcal{A}_{b,c}^{\tiny\textcircled{g}}$.\end{defn}

In Section 2, we characterize generalized $(b,c)$-EP inverse of an element by virtue of a kind of its associated decomposition. We prove that $a\in \mathcal{A}$ has generalized $(b,c)$-EP inverse if and only if there exist $x,y\in \mathcal{A}$ such that
$$a=x+y, x^*y=yx=0, x\in \mathcal{A}~\mbox{has} (b,c)-EP ~\mbox{inverse}, y\in \mathcal{A}^{qnil}.$$ Here, $\mathcal{A}^{qnil}=\{x\in \mathcal{A}~\mid~ \lim\limits_{n\to \infty}\parallel x^n\parallel^{\frac{1}{n}}=0\}.$ As is well known, $x\in \mathcal{A}^{qnil}$ if and only if $1+\lambda x\in \mathcal{A}$ is invertible for any $\lambda\in {\Bbb C}$.

Recall that $a\in \mathcal{A}$ has g-Drazin inverse (i.e., generalized Drazin inverse) if there exists $x\in \mathcal{A}$ such that $$ax^2=x, ax=xa, a-a^2x\in \mathcal{A}^{qnil}.$$ Such $x$ is unique, if exists, and denote it by $a^d$. The g-Drazin inverse plays an important role in matrix and operator theory (see~\cite{CM2}). In Section 3, we are concerned with the generalized EP properties by using the g-Drazin invertibility.
We studied the generalized EP inverse of a Banach element by its polar-like property. We prove that $a\in \mathcal{A}_{b,c}^{\tiny\textcircled{g}}$ if and only if $a\in \mathcal{A}^{d}$ and there exists an idempotent $p\in \mathcal{A}$ such that $$a+p\in \mathcal{A}^{-1}, 1-p\in \mathcal{A}_{b,c}^{\#}, (a^*ap)^*=a^*ap~\mbox{and} ~pa=pap\in \mathcal{A}^{qnil}.$$ An element $a\in \mathcal{A}$ has weak $(b,c)$-EP inverse if $a$ has weak group inverse and $a^{\tiny\textcircled{W}}=a^{(b,c)}$, and we denote it by $a_{(b,c)}^{\tiny\textcircled{W}}$. As an application, we characterize when the weak group inverse of a Banach element is its $(b,c)$-inverse.

In Section 4, we establish the relations of the generalized $(b,c)$-EP inverse and its associated generalized inverses. We prove that
$a\in \mathcal{A}_{b,c}^{\tiny\textcircled{g}}$ if and only if $a\in \mathcal{A}^{\tiny\textcircled{g}}$ and $cab\in \mathcal{A}^{-}$ such that $a^{\tiny\textcircled{g}}=b(cab)^{-}c, b(cab)^{-}(cab)=b, (cab)(cab)^{-}c=c.$ Here, $x^{-}$ denotes the inner inverse of an element $x$ in $\mathcal{A}$.
Finally, we establish the necessary and sufficient conditions under which an element has generalized $(b,b)$-EP inverse and reveal an intrinsic connection between the generalized $(b,b)$-EP inverse and the converse law for the generalized group inverse.

Throughout the paper, all Banach algebras are complex with a proper involution $*$. We use $\mathcal{A}^{\#}, \mathcal{A}^{d}, \mathcal{A}^{\tiny\textcircled{g}}, \mathcal{A}^{\tiny\textcircled{d}}$ and $\mathcal{A}^{\tiny\textcircled{W}}$  to denote the sets of
all group invertible, g-Drazin invertible, generalized group invertible, generalized core invertible and weak group invertible elements in $\mathcal{A}$, respectively.

\section{generalized $(b,c)$-EP inverse}

The purpose of this section is to establish fundamental properties of the generalized group-$(b,c)$-inverse. We start by

\begin{thm} Let $a\in \mathcal{A}$. Then the following are equivalent:\end{thm}
\begin{enumerate}
\item [(1)] $a\in \mathcal{A}$ has generalized $(b,c)$-EP inverse.
\vspace{-.5mm}
\item [(2)] There exist $x,y\in \mathcal{A}$ such that
$$a=x+y, x^*y=yx=0, x\in \mathcal{A}_{b,c}^{\#}, y\in \mathcal{A}^{qnil}.$$
\vspace{-.5mm}
\item [(3)] There exists $x\in \mathcal{A}_{b,c}^{\#}$ such that $$x=ax^2, (a^*a^2x)^*=a^*a^2x, \lim\limits_{n\to \infty}||a^n-xa^{n+1}||^{\frac{1}{n}}=0.$$
\vspace{-.5mm}
\item [(4)] $a\in \mathcal{A}^{d}$ and there exists $x\in \mathcal{A}_{b,c}^{\#}$ such that $$x=ax^2, (a^d)^*a^2x=(a^d)^*a, \lim\limits_{n\to \infty}||a^n-xa^{n+1}||^{\frac{1}{n}}=0.$$
\end{enumerate}
\begin{proof} $(1)\Rightarrow (2)$ By hypothesis, $a\in \mathcal{A}^{\tiny\textcircled{g}}$ and
$a^{\tiny\textcircled{g}}=a^{(b,c)}$. Then we verify that
$$\begin{array}{rll}
a^{\tiny\textcircled{g}}ab&=&a^{\tiny\textcircled{g}}(aa^{\tiny\textcircled{g}}a)b=b,\\
caa^{\tiny\textcircled{g}}&=&c(aa^{\tiny\textcircled{g}}a)a^{\tiny\textcircled{g}}=c.
\end{array}$$
This implies that $a^{\tiny\textcircled{g}}=(aa^{\tiny\textcircled{g}}a)^{(b,c)}$.
In view of~\cite[Theorem 1.1]{C2}, there exist $x,y\in \mathcal{A}$ such that $$a=x+y, x^*y=yx=0, x\in
\mathcal{A}^{\#}, y\in \mathcal{A}^{qnil}.$$ Moreover, we check that
$$x^{\#}=a^{\tiny\textcircled{g}}=(aa^{\tiny\textcircled{g}}a)^{(b,c)}=x^{(b,c)}.
$$ Therefore $x\in \mathcal{A}_{b,c}^{\#}$. Accordingly, we have $a\in \mathcal{A}_{b,c}^{\tiny\textcircled{g}}$.

$(2)\Rightarrow (1)$ Since $a\in \mathcal{A}_{b,c}^{\tiny\textcircled{g}}$,
there exist $x,y\in \mathcal{A}$ such that $$a=x+y, x^*y=yx=0, x\in
\mathcal{A}_{b,c}^{\#}, y\in \mathcal{A}^{qnil}.$$ Then $x\in \mathcal{A}^{\#}$.
In view of ~\cite[Theorem 1.1]{C2}, $a\in \mathcal{A}^{\tiny\textcircled{g}}$ and
$$a^{\tiny\textcircled{g}}=x^{\#}=x^{(b,c)}=(aa^{\tiny\textcircled{g}}a)^{(b,c)}.$$ This implies that
$$\begin{array}{rll}
a^{\tiny\textcircled{g}}ab&=&a^{\tiny\textcircled{g}}(aa^{\tiny\textcircled{g}}a)b=b,\\
caa^{\tiny\textcircled{g}}&=&c(aa^{\tiny\textcircled{g}}a)a^{\tiny\textcircled{g}}=c.
\end{array}$$ Obviously, $a^{\tiny\textcircled{g}}=a^{\tiny\textcircled{g}}aa^{\tiny\textcircled{g}}$.
Therefore $a^{\tiny\textcircled{g}}=a^{(b,c)},$ as desired.

$(2)\Rightarrow (4)$  By hypothesis, $a$ has the generalized $(b,c)$-EP decomposition $a=a_1+a_2$.
Let $x=(a_1)_{b,c}^{\#}$. Then $ax^2=(a_1+a_2)(a_1^{\#})^2=a_1(a_1^{\#})^2=a_1^{\#}=x$.

Since $a_2a_1=0$, we have $a-xa^2=(a_1+a_2)-(a_1^{\#}a_1+a_1^{\#}a_2)(a_1+a_2)=
(1-a_1^{\#}a_1-a_1^{\#}a_2)a_2$, and then
$$\begin{array}{rl}
&||a^n-xa^{n+1}||^{\frac{1}{n}}\\
=&||(a-xa^2)a^{n-1}||^{\frac{1}{n}}\\
=&||(1-a_1^{\#}a_1-a_1^{\#}a_2)a_2(a_1+a_2)^{n-1}||^{\frac{1}{n}}\\
\leq &||1-a_1^{\#}a_1-a_1^{\#}a_2||^{\frac{1}{n}}||a_2^{n}||^{\frac{1}{n}}.
\end{array}$$
As $\lim\limits_{n\to \infty}||(a_2)^{n}||^{\frac{1}{n}}=0$, we derive
$\lim\limits_{n\to \infty}||a^n-xa^{n+1}||^{\frac{1}{n}}=0.$
Clearly, $a_2a_1=0, a_1a_1^{\pi}=0$ and $a_2^d=0$, and then $a\in \mathcal{A}^d$ and $a^d=a_1^{\#}+\sum\limits_{n=1}^{\infty}(a_1^{\#})^{n+1}a_2^n$ (see~\cite[Corollary 3.4]{CK}). We further check that
Then $$\begin{array}{rll}
(a^d)^*a_2&=&(a_1^{\#})^*a_2+\sum\limits_{n=1}^{\infty}[(a_1^{\#})^{n+1}a_2^n]^*a_2\\
&=&[(a_1^{\#})^2]^*(a_1^*a_2)+\sum\limits_{n=1}^{\infty}[(a_1^{\#})^{n+2}a_2^n]^*(a_1^*a_2)\\
&=&0;
\end{array}$$ hence, $(a^d)^*a_1=(a^d)^*a$.
Therefore $(a^d)^*a^2x=(a^d)^*(a_1+a_2)(a_1+a_2)a_1^{\#}=(a^d)^*a_1^2a_1^{\#}=(a^d)^*a_1=(a^d)^*a.$

$(4)\Rightarrow (3)$ By hypothesis, there exists $x\in \mathcal{A}_{b,c}^{\#}$ such that $$x=ax^2, (a^d)^*a^2x=(a^d)^*a, \lim\limits_{n\to \infty}||a^n-xa^{n+1}||^{\frac{1}{n}}=0.$$
Let $a_1=a^2x$ and $a_2=a-a^2x$. Then we verify that
$$\begin{array}{rll}
||a_2a_1||&=&||(a-a^2x)a^2x||=||a^3x-a^2xa^2x||\\
&=&||a^3x-a^2xa^{k+1}x^k||\\
&=&||a^3x-a^{k+2}x^k+a^2(a^k-xa^{k+1})x^k-a^{k+2}x^k||\\
&\leq&||a^3x-a^{k+2}x^k||+||a^2||||a^k-xa^{k+1}||||x^k-a^{k+2}x^k||.
\end{array}$$ Then $$\begin{array}{rll}
||a_2a_1||^{\frac{1}{k}}&\leq &||a^2||^{\frac{1}{k}}||a^k-xa^{k+1}||^{\frac{1}{k}}||x^k-a^{k+2}x^k||^{\frac{1}{k}}\\
&\leq &||a^2||^{\frac{1}{k}}||a^k-xa^{k+1}||^{\frac{1}{k}}(1+||a^2||^{\frac{1}{k}}||a||)||x||.
\end{array}$$ Therefore $\lim\limits_{k\to \infty}||a_2a_1||^{\frac{1}{k}}=0$, and then $a_2a_1=0$.

Moreover, we have
$$\begin{array}{rl}
&||a_1^*a_2||\\
=&||(a^2x)^*(a-a^2x)||=||(a^2x)^*a-(a^2x)^*a^2x||\\
=&||(a^2x)^*a-(a^{k}x^{k-1})^*a^2x||\\
=&||(a^2x)^*a-[(a^{k}-a^da^{k+1})x^{k-1}+a^da^{k+1}x^{k-1}]^*a^2x||\\
\leq&||(a^{k}-a^da^{k+1})^*||(x^{k-1})^*||||a^2x||+||(a^2x)^*a-[a^da^{k+1}x^{k-1}]^*a^2x||\\
\leq&||(a^{k}-a^da^{k+1})^*||(x^{k-1})^*||||a^2x||+||(a^2x)^*a-[a^{k+1}x^{k-1}]^*(a^d)^*a^2x||\\
\leq&||(a^{k}-a^da^{k+1})^*||(x^{k-1})^*||||a^2x||+||(a^2x)^*a-[a^{k+1}x^{k-1}]^*(a^d)^*a||\\
\leq&||(a^{k}-a^da^{k+1})^*||(x^{k-1})^*||||a^2x||+||(a^2x)^*a-[a^da^{k+1}x^{k-1}]^*a||\\
\leq&||(a^{k}-a^da^{k+1})^*||(x^{k-1})^*||||a^2x||+||(a^kx^{k-1})^*a-[a^da^{k+1}x^{k-1}]^*a||\\
\leq&||(a^{k}-a^da^{k+1})^*||(x^{k-1})^*||||a^2x||+||(a^k-a^da^{k+1})x^{k-1}]^*||||a||\\
\leq&||(a^{k}-a^da^{k+1})^*||||(x^{k-1})^*||||a^2x||+||(a^k-a^da^{k+1})^*||||(x^{k-1})^*||||a||.
\end{array}$$ Then $$\begin{array}{rll}
||a_1^*a_2||^{\frac{1}{k}}&\leq &||(a^{k}-a^da^{k+1})^*||^{\frac{1}{k}}||(x^{k-1})^*||^{\frac{1}{k}}||a^2x||^{\frac{1}{k}}\\
&+&||(a^k-a^da^{k+1})^*||^{\frac{1}{k}}||(x^{k-1})^*||^{\frac{1}{k}}||a||^{\frac{1}{k}}.
\end{array}$$ Then $\lim\limits_{k\to \infty}||a_1^*a_2||^{\frac{1}{k}}=0$, and so $a_1^*a_2=0$.

Obviously, we have $$\begin{array}{rll}
a_1x&=&a^2x^2=a(ax^2)=ax=(xa^2)x=x(a^2x)=xa_1,\\
a_1xa_1&=&ax(a^2x)=a^2x=a_1,\\
xa_1x&=&a_1x^2=(a^2x)x^2=a(ax^2)x=ax^2=x.
\end{array}$$ Hence $a_1^{\#}=x$.

Claim 1. $x\in b\mathcal{A}x\bigcap x\mathcal{A}c.$

$x=a_1^{\#}=b\mathcal{A}a_1^{\#}\bigcap a_1^{\#}\mathcal{A}c$.

Claim 2. $xa_1b=b$.

$$x(a^2x)b=axb=x^{\#}xb=b.$$

Claim 3. $ca_1x=c$.

$$ca_1x=c(a^2x)x=cax=cxx^{\#}=c.$$

Therefore $a_1^{\#}=x=(a_1)^{(b,c)}$; and so $a_1\in \mathcal{A}_{b,c}^{\#}$.

Obviously, $a=a^2x+a-a^2x=a_1+a_2$. Then $a^*a^2x=(a_1^*+a_2^*)(a_1+a_2)^2a_1^{\#}=(a_1^*+a_2^*)a_1=a_1^*a_1+(a_1^*a_2)^*=a_1^*a_1$.
Accordingly, $(a^*a^2x)^*=(a_1^*a_1)^*=a_1^*a_1=a^*a^2x$, as required.

$(3)\Rightarrow (2)$ By hypotheses, we have $z\in \mathcal{A}_{b,c}^{\#}$ such that $$z=az^2, (a^*a^2z)^*=a^*a^2z, \lim\limits_{n\to \infty}||a^n-za^{n+1}||^{\frac{1}{n}}=0.$$
For any $n\in {\Bbb N}$, we have $az=a(az^2)=a^2z^2=a^2(az^2)z=a^3z^3=\cdots =a^nz^n=\cdots =a^{n+1}z^{n+1}$.
Thus, we prove that $$\begin{array}{rll}
||z-zaz||&=&||(az)z-zaz||\\
&=&||(a^nz^n)z-z(a^{n+1}z^{n+1})||\\
&=&||(a^n-za^{n+1})z^{n+1}||.
\end{array}$$ Hence, $$\begin{array}{rll}
||z-zaz||^{\frac{1}{n}}&\leq &||(a^n-za^{n+1})||^{\frac{1}{n}}||z||^{\frac{n+1}{n}}.
\end{array}$$
This implies that $\lim\limits_{n\to \infty}||z-zaz||^{\frac{1}{n}}=0,$ whence, $z=zaz$.

Set $x=a^2z$ and $y=a-a^2z.$ Then $a=x+y$.
We check that $$\begin{array}{rcl}
(a-za^2)z&=&(a-za^2)az^2\\
&=&(a-za^2)a^2z^3\\
&\vdots&\\
&=&(a-za^2)a^{n-1}z^n\\
&=&(a^n-za^{n-1})z^n.
\end{array}$$
Therefore $$||(a-za^2)z||^{\frac{1}{n}}\leq ||a^n-za^{n-1}|^{\frac{1}{n}}|||z^n||^{\frac{1}{n}}.$$
Since $\lim\limits_{n\to \infty}||a^n-za^{n+1}||^{\frac{1}{n}}=0,$ we have $\lim\limits_{n\to \infty}||(a-za^2)z||^{\frac{1}{n}}=0.$
This implies that $(a-za^2)z=0$.

We claim that $x$ has group inverse. Evidently, we verify that
$$\begin{array}{rll}
zx^2&=&za^2za^2z=a(za^2z)=a^2z=x,\\
xz^2&=&(a^2z)z^2=(a^2z^2)z=az^2=z,\\
xz&=&a^2z^2=az=za^2z=zx.
\end{array}$$ Hence, $x\in \mathcal{A}^{\#}$ and $x^{\#}=z$.

We check that
$$\begin{array}{rll}
||(a-za^2)^{n+1}||^{\frac{1}{n+1}}&=&||(a-za^2)^n(a-za^2)||^{\frac{1}{n+1}}\\
&=&||(a-za^2)^{n-1}(a-za^2)a||^{\frac{1}{n+1}}\\
&=&||(a-za^2)^{n-1}a^2||^{\frac{1}{n+1}}\\
&\vdots&\\
&=&||(a-za^2)a^n||^{\frac{1}{n+1}}\\
&\leq &\big[||a^n-za^{n+1}||^{\frac{1}{n}}\big]^{\frac{n}{n+1}}||a^n||^{\frac{1}{n+1}}.
\end{array}$$ Accordingly, $\lim\limits_{n\to \infty}||(a-za^2)^{n+1}||^{\frac{1}{n+1}}=0.$ This implies that $a-za^2\in \mathcal{A}^{qnil}$. By using Cline's formula (see~\cite[Theorem 2.1]{L}), $y=a-a^2z\in \mathcal{A}^{qnil}$. Moreover, we see that $$\begin{array}{rll}
x^*y&=&(a^2z)^*(a-a^2z)=[z^*(a^2)^*a](1-az)\\
&=&(a^*a^2z)(1-az)=0,\\
yx&=&(a-a^2z)(a^2z)=a^3z-a^2(za^2z)=a^3z-a^2(az)=0.\end{array}$$ We claim that $x\in \mathcal{A}_{b,c}^{\#}$. In view of ~\cite[Theorem 1.1]{C2},
$z=a^{\tiny\textcircled{g}}$, and so $x=z^{\#}=a^2a^{\tiny\textcircled{g}}$.

Since $z^{\#}=z^{(b,c)}$, we have $z^{\#}\in b\mathcal{A}\bigcap \mathcal{A}c, z^{\#}zb=b$ and $czz^{\#}=c$. Then
$$\begin{array}{rll}
x^{\#}&=&z=z^{\#}zz^{\#}\in b\mathcal{A}\bigcap \mathcal{A}c,\\
x^{\#}xb&=&xx^{\#}b=z^{\#}zb=b,\\
cxx^{\#}&=&cx^{\#}x=czz^{\#}=c.
\end{array}$$ Therefore $x^{\#}=x^{(b,c)},$ as desired.\end{proof}

\begin{cor} Let $a\in \mathcal{A}$. Then the following are equivalent:\end{cor}
\begin{enumerate}
\item [(1)] $a\in \mathcal{A}$ has generalized $(b,c)$-EP inverse.
\vspace{-.5mm}
\item [(2)] There exists a unique element $x\in \mathcal{A}_{b,c}^{\#}$ such that $$x=ax^2, (a^*a^2x)^*=a^*a^2x, \lim\limits_{n\to \infty}||a^n-xa^{n+1}||^{\frac{1}{n}}=0.$$
\end{enumerate}
\begin{proof} $(1)\Rightarrow (2)$ Set $x=a_{b,c}^{\tiny\textcircled{g}}$. Then $x=a^{\tiny\textcircled{g}}$. The uniqueness is proved as the generalized group inverse $x$ is unique.

$(2)\Rightarrow (1)$ This is obvious by Theorem 2.1.\end{proof}

We denote the unique $x$ in Corollary 2.2 by $a_{b,c}^{\tiny\textcircled{g}}$, and call it the generalized group-$(b,c)$-inverse of $a$. Let $\mathcal{A}_{b,c}^{\tiny\textcircled{g}}$ denote the sets of all generalized group-$(b,c)$-invertible elements in $\mathcal{A}$.

\begin{cor} Let $a\in \mathcal{A}_{b,c}^{\tiny\textcircled{g}}$ and $d\in \mathcal{A}^{qnil}$. If $a^*d=da=0$, then
$a+d\in \mathcal{A}_{b,c}^{\tiny\textcircled{g}}$. In this case, $(a+d)_{b,c}^{\tiny\textcircled{g}}=a_{b,c}^{\tiny\textcircled{g}}.$\end{cor}
\begin{proof} Since $a\in \mathcal{A}_{b,c}^{\tiny\textcircled{g}}$, by virtue of Theorem 2.1, there exist $x\in \mathcal{A}_{b,c}^{\#}$ and $y\in \mathcal{A}^{qnil}$ such that $a=x+y, x^*y=yx=0$. As in the proof of Theorem 2.1, $x=a^2a_{b,c}^{\#}$ and $y=a-a^2a_{b,c}^{\#}$.
Then $a+d=x+(y+d)$. As $dy=d(a-a^2a_{b,c}^{\#})=0$, it follows by ~\cite[Lemma 2.4]{CK} that $y+d\in \mathcal{A}^{qnil}$.
Obviously, $x^*(y+d)=x^*y+x^*d=0$ and $(y+d)x=yx+dx=0$.
In light of Theorem 2.1, $a+d\in \mathcal{A}_{b,c}^{\tiny\textcircled{g}}$. In this case, $(a+d)^{\tiny\textcircled{g}}=x_{b,c}^{\tiny\textcircled{\#}}=a_{b,c}^{\tiny\textcircled{g}}.$\end{proof}

\begin{exam} Let $a\in \mathcal{A}^{\tiny\textcircled{g}}$ and $d\in \mathcal{A}^{qnil}$. If $a^*d=da=0$, then
$a+d\in \mathcal{A}_{b,c}^{\tiny\textcircled{g}}$, where $b=a^{\tiny\textcircled{g}}a$ and
$c=aa^{\tiny\textcircled{g}}$. \end{exam}
\begin{proof} In view of ~\cite[Theorem 3.6]{C4}, $a^{\tiny\textcircled{g}}=a^{\tiny\textcircled{g}}aa^{\tiny\textcircled{g}}a$. We further check that
$$\begin{array}{rll}
a^{\tiny\textcircled{g}}&\in &b\mathcal{A}\bigcap \mathcal{A}c,\\
a^{\tiny\textcircled{g}}ab&=&a^{\tiny\textcircled{g}}aa^{\tiny\textcircled{g}}a=b,\\
caa^{\tiny\textcircled{g}}&=&aa^{\tiny\textcircled{g}}aa^{\tiny\textcircled{g}}=c.
\end{array}$$ Therefore $a^{\tiny\textcircled{g}}=a^{(b,c)}$, i.e.,
$a\in \mathcal{A}_{b,c}^{\tiny\textcircled{g}}$. According to Corollary 2.3, $a+d\in \mathcal{A}_{b,c}^{\tiny\textcircled{g}}$. In this case, $(a+d)_{b,c}^{\tiny\textcircled{g}}=a_{b,c}^{\tiny\textcircled{g}}.$
\end{proof}

\begin{exam} \end{exam} Let $$\begin{array}{c}
A=\left(
\begin{array}{ccc}
1&i&0\\
0&0&1\\
i&-1&0
\end{array}
\right), B=\left(
\begin{array}{ccc}
1&i&0\\
0&0&1\\
i&-1&0
\end{array}
\right),\\
C=\left(
\begin{array}{ccc}
1&0&-i\\
-i&0&-1\\
0&1&0
\end{array}
\right),D=\left(
\begin{array}{ccc}
-i&0&1\\
0&0&0\\
1&0&i
\end{array}
\right)\in {\Bbb C}^{3\times 3}.
\end{array}$$ Then $A^{\dag}=\left(
\begin{array}{ccc}
0&0&\frac{1}{2\sqrt{2}}\\
0&0&\frac{i}{2\sqrt{2}}\\
0&1&0
\end{array}
\right)$. One easily checks that
$$R(A)\subseteq R(B)\bigcap R(C^*), (I_3-AA^{\dag})B=C(I_3-A^{\dag}A)=0.$$ In view of ~\cite[Theorem 1]{H}, $AA^*\in \mathcal{A}^{\#}$ and $(AA^*)^{\#}=(A^{\dag})^*A^{\dag}.$
we verify that
$$\begin{array}{rll}
(AA^*)^{\#}&=&(AA^*)^{\#}(AA^*)(AA^*)^{\#},\\
B&=&(AA^*)^{\#}(AA^*)B,\\
C&=&C(AA^*)(AA^*)^{\#},\\
(AA^*)^{\#}&\in &B{\Bbb C}^{3\times 3}\bigcap {\Bbb C}^{3\times 3}C.
\end{array}$$ Therefore $AA^*\in \mathcal{A}_{B,C}^{\#}$. Set $$M:=AA^*+D=\left(
\begin{array}{ccc}
2-i & 0 & 1-2i \\
0 & 1 & 0 \\
1+2i & 0 & 2+i
\end{array}
\right).$$ Clearly, $A^*D=DA=0$. Since $(AA^*)^*D=D(AA^*)=0$ and $D$ is nilpotent,
by virtue of Theorem 2.1, $M\in {\Bbb C}^{3\times 3}$ has weak group-$(B,C)$-inverse. In this case, we have $$M_{B,C}^{\tiny\textcircled{W}}=(A^{\dag})^*A^{\dag}=\left(
\begin{array}{ccc}
0&\frac{1}{2\sqrt{2}}&0\\
0&\frac{i}{2\sqrt{2}}&0\\
0&0&\frac{i}{2\sqrt{2}}
\end{array}
\right).$$\\

\begin{lem} Let $a\in \mathcal{A}$. Then the following are equivalent:\end{lem}
\begin{enumerate}
\item [(1)] $a\in \mathcal{A}_{b,c}^{\#}$.
\item [(2)] $a\in \mathcal{A}^{(b,c)}$ and $\ell(b)=\ell(a), r(c)\subseteq r(a).$
\item [(3)] $a\in \mathcal{A}^{(b,c)}$ and $\ell(b)=\ell(a), r(c)=r(a).$
\end{enumerate}
\begin{proof} $(1)\Rightarrow (3)$ By virtue of ~\cite[Proposition 3.10]{WC}, we have $a\in \mathcal{A}^{(b,c)}$ and $a\mathcal{A}=b\mathcal{A}, \mathcal{A}a=\mathcal{A}c.$ Therefore we verify that $\ell(a)=\ell(b), r(a)=r(c).$

$(3)\Rightarrow (2)$ This is trivial.

$(2)\Rightarrow (1)$ Set $x=a^{(b,c)}$. Then $xab=b$ and $cax=c$. Hence, $1-xa\in \ell(b)$; and so
$1-xa\in \ell(a)$. This implies that $a=xa^2$. Moreover, $1-ax\in r(c)$, and then $a(1-ax)=0$. This implies that $a=a^2x$.
We infer that $a\in \mathcal{A}^{\#}$. Hence, $aa^{\#}=ax$ and $xa=aa^{\#}$. This implies that $x=x(ax)=x(aa^{\#})=(xa)a^{\#}=(aa^{\#})a^{\#}=
a^{\#}$. Therefore $a^{\#}=a^{(b,c)}$, as asserted.\end{proof}

We come now to investigate the generalized $(b,c)$-EP inverse by using the related annihilators.

\begin{thm} Let $a\in \mathcal{A}^{\tiny\textcircled{g}}$. Then the following are equivalent:\end{thm}
\begin{enumerate}
\item [(1)] $a\in \mathcal{A}_{b,c}^{\tiny\textcircled{g}}$.
\item [(2)] $a\in \mathcal{A}^{\tiny\textcircled{g}}\bigcap\mathcal{A}^{(b,c)}$ and $\ell(b)=\ell(a^{\tiny\textcircled{g}}), r(c)\subseteq r(a^{\tiny\textcircled{g}}).$
\item [(3)] $a\in \mathcal{A}^{\tiny\textcircled{g}}\bigcap\mathcal{A}^{(b,c)}$ and $\ell(b)=\ell(a^{\tiny\textcircled{g}}), r(c)=r(a^{\tiny\textcircled{g}}).$
\end{enumerate}
\begin{proof} $(1)\Rightarrow (3)$ In view of Theorem 2.1, there exist $x,y\in \mathcal{A}$ such that
$$a=x+y, x^*y=yx=0, x\in \mathcal{A}_{b,c}^{\#}, y\in \mathcal{A}^{qnil}.$$ Evidently, $x=a^2a^{\tiny\textcircled{g}}\in R^{\#}$.
Then $a^2a^{\tiny\textcircled{g}}\in \mathcal{A}^{(b,c)}$. By virtue of Lemma 2.6, we have $\ell(b)=\ell(x)$ and $r(c)=r(x)$.
It is easy to verify that $\ell(x)=\ell(a^2a^{\tiny\textcircled{g}})=\ell(a^{\tiny\textcircled{g}})$ and
$r(x)=r(a^2a^{\tiny\textcircled{g}})=r(aa^{\tiny\textcircled{g}})=r(a^{\tiny\textcircled{g}})$, as desired.

$(3)\Rightarrow (2)$ This is trivial.

$(2)\Rightarrow (1)$ As in the preceding discussion, we have $\ell(b)=\ell(a^{\tiny\textcircled{g}})=\ell(a^2a^{\tiny\textcircled{g}})$ and
$r(c)\subseteq r(a^{\tiny\textcircled{g}})=r(a^2a^{\tiny\textcircled{g}}).$ Since $a^2a^{\tiny\textcircled{g}}\in R^{\#}$,
it follows by ~\cite[Theorem 5.1.5, Theorem 5.1.6]{WC1} that $a^2a^{\tiny\textcircled{g}}\in R_{b,c}^{\#}$. As in the proof in ~\cite[Theorem 2.2]{C7}, we have a generalized $(b,c)$-EP decomposition
$a=x+y$, where $x=a^2a^{\tiny\textcircled{g}}$. Therefore $a\in \mathcal{A}_{b,c}^{\tiny\textcircled{g}}$.\end{proof}

An element $a$ is EP if $a^{\#}=a^{\dag}$, where $a^{\dag}$ is the Moore-Penrose inverse. An element $a$ in $\mathcal{A}$ is generalized EP (i.e., a generalized EP element) if there exist $x,y\in \mathcal{A}$ such that $a=x+y, x^*y=yx=0, x\in \mathcal{A}~\mbox{is EP},~y\in \mathcal{A}^{qnil}$ (see~\cite{C5}).
We characterize the generalized EP element by the generalized $(b,c)$-EP invertibility.

\begin{thm} Let $a\in \mathcal{A}$. Then the following are equivalent:\end{thm}
\begin{enumerate}
\item [(1)] $a\in \mathcal{A}$ is generalized EP.
\item [(2)] $a\in \mathcal{A}_{(a^d)^*,a^d}^{\tiny\textcircled{g}}$.
\item [(3)] $a\in \mathcal{A}_{a^d,(a^d)^*}^{\tiny\textcircled{g}}$.
\end{enumerate}
\begin{proof} $(1)\Rightarrow (2)$ By hypothesis, there exist $x,y\in \mathcal{A}$ such that
$$a=x+y, x^*y=yx=0, x\in \mathcal{A}~\mbox{is EP}, ~y\in \mathcal{A}^{qnil}.$$
Then $a^d=x^{\#}+\sum\limits_{n=1}^{\infty}(x^{\#})^{n+1}y^n$.
Since $x$ is EP, we see that $xx^{\#}=x^{\#}x=(xx^{\#})^*$. Hence, $x^{\#}y=x(xx^{\#})y=(xx^{\#})^*y=(x^{\#})^*(x^*y)=0$.
Hence, $a^d=x^{\#}$. Then $x^{\#}\in \mathcal{A}^d$.
Moreover, we have $x^{\#}=xx^{\#}x=(xx^{\#})^*x=(a^d)^*x^*x\in (a^d)^*\mathcal{A}$.
Thus we derive that
$$\begin{array}{rll}
a^{\tiny\textcircled{g}}&=&x^{\#}=a^{\tiny\textcircled{d}}\in (a^d)^*\mathcal{A}\bigcap \mathcal{A}a^d,\\
a^{\tiny\textcircled{g}}a(a^d)^*&=&x^{\#}(x+y)(x^{\#})^*=(x^{\#}xx^{\#})^*=(x^{\#})^*=(a^d)^*,\\
a^daa^{\tiny\textcircled{g}}&=&x^{\#}(x+y)x^{\#}=x^{\#}=a^d.
\end{array}$$ Therefore $a^{\tiny\textcircled{g}}=a_{(a^d)^*,a^d}^{\tiny\textcircled{g}}$, as required.

$(2)\Rightarrow (1)$ By hypothesis, we have $$\begin{array}{rll}
a^{\tiny\textcircled{g}}&=&a^{\tiny\textcircled{g}}(aa^{\tiny\textcircled{g}}a)a^{\tiny\textcircled{g}},\\
(a^d)^*&=&a^{\tiny\textcircled{g}}(aa^{\tiny\textcircled{g}}a)(a^d)^*,\\
a^d&=&a^d(aa^{\tiny\textcircled{g}}a)a^{\tiny\textcircled{g}},\\
\end{array}$$ For any $n\in {\Bbb N}$, we derive $$\begin{array}{rll}
a^{\tiny\textcircled{g}}-a^d&=&a^n(a^{\tiny\textcircled{g}})^{n+1}-a^daa^{\tiny\textcircled{g}}\\
&=&a^n(a^{\tiny\textcircled{g}})^{n+1}-a^da^{n+1}(a^{\tiny\textcircled{g}})^{n+1}\\
&=&[a^n-a^da^{n+1}](a^{\tiny\textcircled{g}})^{n+1}.
\end{array}$$ This implies that $a^{\tiny\textcircled{g}}=a^d$. Hence $a^da(a^d)^*=(a^d)^*$, and then $a^d=a^d(aa^d)^*$.
Hence $aa^d=aa^d(aa^d)^*$. Thus $a^{\tiny\textcircled{d}}=a^{\tiny\textcircled{g}}=a^d$. According to \cite[Theorem 2.8]{C5},
$a\in \mathcal{A}$ is generalized EP, as desired.

$(1)\Leftrightarrow (3)$ This is proved analogously to the argument above.\end{proof}

\begin{cor} Let $a\in \mathcal{A}$. Then the following are equivalent:\end{cor}
\begin{enumerate}
\item [(1)] $a\in \mathcal{A}$ is *-DMP.
\item [(2)] $a\in \mathcal{A}_{(a^D)^*,a^D}^{\tiny\textcircled{W}}$.
\item [(3)] $a\in \mathcal{A}_{a^D,(a^D)^*}^{\tiny\textcircled{W}}$.
\end{enumerate}
\begin{proof} Since $a\in \mathcal{A}$ is *-DMP if and only if $a\in \mathcal{A}$ is both generalized EP and Drazin invertible.
Therefore we complete the proof by Theorem 2.8.\end{proof}

\section{representations and polar-like properties}

In this section, we characterize the generalized $(b,c)$-EP inverse in terms of its generalized Drazin invertibility and establish its polar-like properties.
We now derive:

\begin{thm} Let $a\in \mathcal{A}$. Then $a\in \mathcal{A}_{b,c}^{\tiny\textcircled{g}}$ if and only if\end{thm}
\begin{enumerate}
\item [(1)] $a\in \mathcal{A}^d$;
\vspace{-.5mm}
\item [(2)] There exists some $x\in \mathcal{A}^{(b,c)}$ such that $(a^d)^*a^dx=(a^d)^*a.$
\end{enumerate}
In this case, $a_{b,c}^{\tiny\textcircled{g}}=(a^d)^3x.$
\begin{proof} $\Longrightarrow $ By virtue of Theorem 2.1, $a\in \mathcal{A}^{d}$ and there exists $z\in \mathcal{A}_{b,c}^{\#}$ such that $$z=az^2, (a^d)^*a^2z=(a^d)^*a, \lim\limits_{n\to \infty}||a^n-za^{n+1}||^{\frac{1}{n}}=0.$$ Choose $x=a^3z$. Then

Claim 1. $x^{\#}\in b\mathcal{A}x^{\#}\bigcap x^{\#}\mathcal{A}c$.

Claim 2. $x^{\#}xb=b$.

Claim 3. $cxx^{\#}=c$.

Moreover, we verify that
$$\begin{array}{rl}
&||(a^d)^*a-(a^d)^*a^dx||\\
=&||(a^d)^*a-(a^d)^*a^da^3z||\\
=&||(a^d)^*a-(a^d)^*a^da^{k+1}z^{k-1}||\\
\leq&||(a^d)^*a-(a^d)^*a^{k}z^{k-1}||+||(a^d)^*(a^{k}-a^da^{k+1})z^{k-1}||\\
\leq&||(a^d)^*a-(a^d)^*a^2z||+||(a^d)^*||||a^{k}-a^da^{k+1}||||z^{k-1}||\\
=&||(a^d)^*||||a^{k}-a^da^{k+1}||||z^{k-1}||.
\end{array}$$ Since $\lim\limits_{k\to \infty}||a^k-a^da^{k+1}||^{\frac{1}{k}}=0$, we derive that
$\lim\limits_{k\to \infty}||(a^d)^*a-(a^d)^*a^dx||^{\frac{1}{k}}=0.$ Hence $(a^d)^*a^dx=(a^d)^*a,$ as required.

$\Longleftarrow $ By hypothesis, $(a^d)^*a^dx=(a^d)^*a$ for some $x\in \mathcal{A}$.
Then $(aa^d)^*aa^d=a^*[(a^d)^*a]a^d=a^*[(a^d)^*a^dx]a^d=(aa^d)^*aa^d(a^dxa^d)$. Since the involution $*$ is proper, we get
$aa^d=a^dxa^d$. Choose $y=(a^d)^3x$. Then we verify that
$$\begin{array}{rll}
ay^2&=&a(a^d)^3x(a^d)^3x=(a^d)^2x(a^d)^3x=a^d(a^dxa^d)(a^d)^2x\\
&=&a^d(aa^d)(a^d)^2x=(a^d)^3x=y,\\
(a^d)^*a^2y&=&(a^d)^*a^2(a^d)^3x=(a^d)^*a^dx=(a^d)^*a.\\
\end{array}$$
Moreover, we see that
$$\begin{array}{rl}
&||a^k-ya^{k+1}||\\
=&||(a^k-a^da^{k+1})+(a^d)^2aa^da^{k+2}-(a^d)^3xa^{k+1}||\\
\leq &||(a^k-a^da^{k+1})||+||(a^d)^2aa^da^{k+2}-(a^d)^3xa^{k+1}||\\
\leq&||(a^k-a^da^{k+1})||+||(a^d)^2a^dxa^da^{k+2}-(a^d)^3xa^{k+1}||\\
\leq&||a^k-a^da^{k+1}||+||(a^d)^3x||||a^{k}-a^da^{k+1}||||a||\\
=&||a^k-a^da^{k+1}||(1+||(a^d)^3x||||a||).
\end{array}$$ Since $\lim\limits_{k\to \infty}||a^k-a^da^{k+1}||^{\frac{1}{k}}=0$, we deduce that
$\lim\limits_{k\to \infty}||a^k-ya^{k+1}||^{\frac{1}{k}}=0.$ In view of Theorem 3.1, $a\in \mathcal{A}_{b,c}^{\tiny\textcircled{g}}$. In this case,
$a_{b,c}^{\tiny\textcircled{g}}=(a^d)^3x.$ This completes the proof.\end{proof}

\begin{cor} Let $a\in \mathcal{A}$. Then $a\in \mathcal{A}_{b,c}^{\tiny\textcircled{g}}$ if and only if\end{cor}
\begin{enumerate}
\item [(1)] $a\in \mathcal{A}^d$;
\vspace{-.5mm}
\item [(2)] There exists an idempotent $q\in \mathcal{A}_{b,c}^{\#}$ such that $$a^d\mathcal{A}=q\mathcal{A}~\mbox{and} ~(a^d)^*aq=(a^d)^*a.$$
\end{enumerate}
\begin{proof} $\Longrightarrow $ By using ~\cite[Theorem 3.5]{C7}, $a\in \mathcal{A}^{d}$ and there exists an idempotent $q\in \mathcal{A}$ such that $$a^d\mathcal{A}=q\mathcal{A}~\mbox{and} ~a^*aq=q^*a^*a.$$ Explicit, $q=aa_{b,c}^{\tiny\textcircled{g}}$. We directly check that
$$\begin{array}{rll}
qaa^d&=&aa_{b,c}^{\tiny\textcircled{g}}aa^d=aa^d,\\
aa^dq&=&aa^daa_{b,c}^{\tiny\textcircled{g}}=q.
\end{array}$$ Thus, $q\in \mathcal{A}_{b,c}^{\#}$. Furthermore, we obtain
$$\begin{array}{rll}
(a^d)^*aq&=&[a(a^d)^2]^*aq=[(a^d)^2]^*(a^*aq)=[(a^d)^2]^*(q^*a^*a)\\
&=&[aq(a^d)^2]^*a=[a(qaa^d)(a^d)^2]^*a=[a(aa^d)(a^d)^2]^*a=(a^d)^*a.
\end{array}$$ We conclude that $a^d\mathcal{A}=q\mathcal{A}$ and $(a^d)^*aq=(a^d)^*a,$ as required.

$\Longleftarrow $ By hypothesis, there exists an idempotent $q\in \mathcal{A}_{b,c}^{\#}$ such that $$a^d\mathcal{A}=q\mathcal{A}~\mbox{and} ~(a^d)^*aq=(a^d)^*a.$$  Write $q=a^dz$ with $z\in \mathcal{A}$. Choose $x=az$. Then $(a^d)^*a^dx=(a^d)^*a^daz=(a^d)^*aq=(a^d)^*a$,
the result follows by Theorem 3.1.\end{proof}

We now proceed to elucidate the generalized $(b,c)$-EP inverse through its characterization by the polar-like property.

\begin{thm} Let $a\in \mathcal{A}$. Then the following are equivalent:\end{thm}
\begin{enumerate}
\item [(1)]{\it $a\in \mathcal{A}_{b,c}^{\tiny\textcircled{g}}$.}
\item [(2)]{\it $a\in \mathcal{A}^{d}$ and there exists an idempotent $p\in \mathcal{A}$ such that $$a+p\in \mathcal{A}^{-1}, 1-p\in \mathcal{A}_{b,c}^{\#}, (a^*ap)^*=a^*ap~\mbox{and} ~pa=pap\in \mathcal{A}^{qnil}.$$}
\item [(3)]{\it $a\in \mathcal{A}^{d}$ and there exists an idempotent $p\in \mathcal{A}$ such that $$a+p\in \mathcal{A}^{-1}, 1-p\in \mathcal{A}_{b,c}^{\#}, (a^d)^*ap=0~\mbox{and} ~pa=pap\in \mathcal{A}^{qnil}.$$}
\end{enumerate}
\begin{proof} $(1)\Rightarrow (2)$ In view of Theorem 2.1, $a\in \mathcal{A}^{d}$. Since $a\in R_{b,c}^{\tiny\textcircled{g}}$, there exist $z,y\in \mathcal{A}$ such that $$a=z+y, z^*y=yz=0, z\in \mathcal{A}_{b,c}^{\#}, y\in \mathcal{A}^{qnil}.$$ Set $x=z^{\#}$. Then we check that $$\begin{array}{rll}
ax&=&(z+y)z^{\#}=zz^{\#},\\
ax^2&=&(ax)x=z(z^{\#})^2=z^{\#}=x,\\
(a^*a^2x)^*&=&(z^*z)^*=z^*z=a^*a^2x.
\end{array}$$ Let $p=1-zz^{\#}$. Then $p=1-ax=p^2\in \mathcal{A}$.

Since $z\in \mathcal{A}_{b,c}^{\#}$, we easily check that $(1-p)b=b$ and $c(1-p)=c$ and $1-p\in b\mathcal{A}\bigcap \mathcal{A}c$.
Then $1-p=(1-p)^{(b,c)}$, and so $1-p\in \mathcal{A}_{b,c}^{\#}$. Furthermore, $a-azz^{\#}=a(1-zz^{\#})=(z+y)(1-zz^{\#})=y\in \mathcal{A}^{qnil}$, and so
$pa=(1-zz^{\#})a=a-zz^{\#}a\in \mathcal{A}^{qnil}$ by Cline's formula (see~~\cite[Theorem 2.1]{L}).
Therefore we have $$(a^*ap)^*=(a^*a-a^*a^2x)^*=a^*a-a^*a^2x=a^*ap.$$
Since $pa(1-p)=(1-zz^{\#})(z+y)zz^{\#}=0$, we get $pa=pap$.
Obviously, $z+1-zz^{\#}=(z^{\#}+1-zz^{\#})^{-1}\in \mathcal{A}^{-1}$.
Since $y(z^{\#}+1-zz^{\#})=y\in \mathcal{A}^{qnil}$, it follows by Cline's formula that $(z^{\#}+1-zz^{\#})y\in \mathcal{A}^{qnil}$.
Hence $1+(z^{\#}+1-zz^{\#})y\in \mathcal{A}^{-1}$. This implies that $$\begin{array}{rll}
a+p&=&z+y+1-zz^{\#}\\
&=&(z+1-zz^{\#})[1+(z^{\#}+1-zz^{\#})^{-1}y]\in\mathcal{A}^{-1},
\end{array}$$ as desired.

$(2)\Rightarrow (1)$ By hypothesis, there exists an idempotent $p\in \mathcal{A}$ such that $$u:=a+p\in \mathcal{A}^{-1}, 1-p\in \mathcal{A}_{b,c}^{\#}, (a^*ap)^*=a^*ap, pa=pap\in \mathcal{A}^{qnil}.$$ Set $x=a(1-p)$ and $y=ap$. Then $a=x+y$. Since $pa=pap\in \mathcal{A}^{qnil}$, we have $y=ap\in \mathcal{A}^{qnil}$ by Cline's formula.
We also see that $pa(1-p)=0$, and then $yx=apa(1-p)=0$. Moreover, we have $x^*y=[a(1-p)]^*ap=(1-p)^*(a^*ap)=(1-p)^*(a^*ap)^*=(1-p^*)p^*a^*a=0$.
It will suffice to prove $x\in \mathcal{A}^{\#}$.

Clearly, $x=a(1-p)=u(1-p)$. Let $z=(1-p)u^{-1}$. Then we check that $zx=1-p$; hence, $zx^2=(1-p)x=(1-p)a(1-p)=a(1-p)=x$.
Since $a\in \mathcal{A}^d, ap\in \mathcal{A}^d$ and $pa(1-p)=0$, it follows by ~\cite[Lemma 2.2]{ZM} that $x=a(1-p)\in \mathcal{A}^d$.
Therefore $x=zx^2=z^nx^{n+1}=z^n(x^n-x^dx^{n+1})x+z^nx^dx^{n+2}$ and $x^2x^d=(zx^2)xx^d=z^nx^{n+2}x^d$.
Then $$\begin{array}{rll}
||x-x^2x^d||^{\frac{1}{n}}&=&||z^n(x^n-x^dx^{n+1})x||^{\frac{1}{n}}\\
&\leq &||z^n||^{\frac{1}{n}}||x^n-x^dx^{n+1}||^{\frac{1}{n}}||x||^{\frac{1}{n}}.
\end{array}$$ Accordingly, $$\lim\limits_{n\to \infty}||x-x^2x^d||^{\frac{1}{n}}=0.$$
Therefore we have $x=x^2x^d$, and so $x\in \mathcal{A}^{\#}$. Hence, $zxx^d=(zx^2)(x^d)^2=x(x^d)^2=x^d=x^{\#}$.
This implies that $xx^{\#}=z(xx^dx)=zx=1-p$. Since $x=a(1-p)=(1-p)a(1-p)$, we see that $x^{\#}=x^d\in (1-p)\mathcal{A}(1-p)$. By hypothesis,
$1-p=(1-p)^{(b,c)}$. Hence, $1-p\in b\mathcal{A}\bigcap \mathcal{A}c, (1-p)b=b$ and $c(1-p)=c$.
This implies that $x^{\#}\in b\mathcal{A}\bigcap \mathcal{A}c$. Moreover, we verify that
$$xx^{\#}=1-p,$$ and then $x^{\#}xb=b$ and $cxx^{\#}=c$. Thus, $x\in \mathcal{A}_{b,c}^{\#}$.
Therefore $a=x+y$ is a generalized $(b,c)$-EP decomposition of $a$. In light of Theorem 2.1, $a\in \mathcal{A}_{b,c}^{\tiny\textcircled{g}}$.

$(1)\Rightarrow (3)$ Clearly, $a\in \mathcal{A}^{d}$. By hypothesis, $a$ has the generalized $(b,c)$-EP decomposition $a=z+y$ with $z\in \mathcal{A}^{\#}_{b,c}$.
Let $x=z^{\#}$ and $p=1-ax$. Then $1-p=ax=(z+y)x=zz^{\#}$; hence, $1-p\in \mathcal{A}_{b,c}^{\#}$. As in the preceding discussion, we prove that
$$a+p\in \mathcal{A}^{-1}~\mbox{and} ~pa=pap\in \mathcal{A}^{qnil}.$$
Similarly to Theorem 2.2, we show that $a\in \mathcal{A}^d$ and $(a^d)^*a^2x=(a^d)^*a.$ Therefore $(a^d)^*ap=(a^d)^*a-(a^d)^*a^2x=0$, as desired.

$(3)\Rightarrow (1)$ By hypothesis, we have an idempotent $p\in \mathcal{A}$ such that $$a+p\in \mathcal{A}^{-1},
1-p\in \mathcal{A}_{b,c}^{\#}, (a^d)^*ap=0~\mbox{and} ~pa=pap\in \mathcal{A}^{qnil}.$$
Set $x=a(1-p)$ and $y=ap$. Then $a=x+y$. Analogously to the preceding discussion, we have $yx=0, x\in \mathcal{A}^{\#}$ and $y\in \mathcal{A}^{qnil}$.
By using Cline's formula, $ap\in \mathcal{A}^{qnil}$. Since $pa(1-p)=0$, it follows by~\cite[Lemma 2.2]{ZM} that $a(1-p)\in \mathcal{A}^d$ and $[a(1-p)]^d=a^d(1-p)$.
Thus we verify that $$\begin{array}{rll}
x^*y&=&(x^dx^2)^*ap\\
&=&(x^2)^*(1-p)^*(a^d)^*=0.
\end{array}$$
As in the preceding discussion, we have $xx^{\#}=1-p$, and then we verify that $x\in \mathcal{A}^{\#}_{b,c}$. According to Theorem 2.1, $a\in \mathcal{A}_{b,c}^{\tiny\textcircled{g}}$.\end{proof}

We next investigate the weak $(b,c)$-EP inverse relative to a pair of elements in a Banach *-algebra. The Drazin inverse of an element $a$ is the unique element $a^D$ satisfying $a^{k+1}a^D = a^k$, $a^D aa^D = a^D$, and $aa^D = a^D a$. The least such non-negative integer $k$ is called the index of $a$, denoted by $\operatorname{ind}(a)$.

\begin{lem} Let $a\in \mathcal{A}$. Then $a\in \mathcal{A}_{b,c}^{\tiny\textcircled{W}}$ if and only if $a\in \mathcal{A}_{b,c}^{\tiny\textcircled{g}}\bigcap \mathcal{A}^{D}$.\end{lem}
\begin{proof} $(1)\Rightarrow (2)$ Similarly to Theorem 2.1, we can find $x,y\in \mathcal{A}$ such that $$a=x+y, x^*y=yx=0, x\in \mathcal{A}_{b,c}^{\#},~y\in \mathcal{A}^{nil}.$$ By virtue of Theorem 2.1, $a\in \mathcal{A}_{b,c}^{\tiny\textcircled{g}}$.
Since $x,y\in \mathcal{A}^D$ and $yx=0$, we have $a\in \mathcal{A}^D$, as required.

$(2)\Rightarrow (1)$ By virtue of Theorem 2.1, there exist $z,y\in \mathcal{A}$ such that
$$a=z+y, z^*y=yz=0, z\in \mathcal{A}_{b,c}^{\#},~y\in \mathcal{A}^{qnil}.$$
Evidently, $z=a^2a^{\tiny\textcircled{g}}$ and $y=a-a^2a^{\tiny\textcircled{g}}$. Since $a\in \mathcal{A}^D$, we see that $a^m=a^Da^{m+1}$ for some $m\in {\Bbb N}$. This implies that $a-aa^Da$ is nilpotent. We easily check that $$\begin{array}{rll}
(a-aa^Da)a^{\tiny\textcircled{g}}&=&aa^{\tiny\textcircled{g}}-aa^Daa^{\tiny\textcircled{g}}\\
&=&aa^{\tiny\textcircled{g}}-aa^Da^m(a^{\tiny\textcircled{g}})^m\\
&=&aa^{\tiny\textcircled{g}}-a^m(a^{\tiny\textcircled{g}})^m=0.
\end{array}$$ Hence $[aa^{\tiny\textcircled{g}}(aa^Da-a)]^2=0$.
Since $(a-aa^Da)[aa^{\tiny\textcircled{g}}(aa^Da-a)]=0$, we see that
$$y=(a-aa^Da)+aa^{\tiny\textcircled{g}}(aa^Da-a)\in \mathcal{A}^{qnil}.$$ As in the proof of Theorem 2.1,
we prove that $a\in \mathcal{A}_{b,c}^{\tiny\textcircled{W}}$, as asserted.\end{proof}

\begin{thm} Let $a\in \mathcal{A}$. Then the following are equivalent:\end{thm}
\begin{enumerate}
\item [(1)] $a\in \mathcal{A}_{b,c}^{\tiny\textcircled{W}}$.
\item [(2)] There exist $x\in \mathcal{A}$ and $m\in {\Bbb N}$ such that $$x=xax, x\mathcal{A}=b\mathcal{A}=a^m\mathcal{A}=a^{m+1}\mathcal{A}, a^*a^m\mathcal{A}\subseteq x^*\mathcal{A}=c^*\mathcal{A}.$$
\item [(3)] There exist $x,y\in \mathcal{A}$ such that $$a=x+y, x^*y=yx=0, x\in
\mathcal{A}_{b,c}^{\#}, y\in \mathcal{A}^{nil}.$$
\item [(4)] $a\in \mathcal{A}^{\tiny\textcircled{W}}$ and
$a^{\tiny\textcircled{W}}=(aa^{\tiny\textcircled{W}}a)^{(b,c)}$.
\vspace{-.5mm}
\item [(5)] There exists an idempotent $p\in \mathcal{A}$ such that
$$a+p\in \mathcal{A}^{-1}, 1-p\in \mathcal{A}_{b,c}^{\#}, (a^d)^*ap=0~\mbox{and} ~ap=pap\in \mathcal{A}^{nil}.$$.
\end{enumerate}
\begin{proof} $(1)\Rightarrow (2)$ Since $a\in \mathcal{A}_{b,c}^{\tiny\textcircled{W}}$, we see that $a\in \mathcal{A}_{b,c}^{\tiny\textcircled{W}}$.
Set $x=a^{\tiny\textcircled{W}}$. It follows by ~\cite[Theorem 3.11]{Z1} that there exists $m\in {\Bbb N}$ such that $$x=xax, x\mathcal{A}=a^m\mathcal{A}=a^{m+1}\mathcal{A}, a^*a^m\mathcal{A}\subseteq x^*\mathcal{A}.$$ As $x=a^{(b,c)}$, we see that
$x\mathcal{A}=b\mathcal{A}$ and $\mathcal{A}x=\mathcal{A}c$. Therefore $x^*\mathcal{A}=c^*\mathcal{A}$, as required.

$(2)\Rightarrow (1)$ By virtue of ~\cite[Theorem 3.11]{Z1}, $x=a^{\tiny\textcircled{W}}$. Since
$x=xax, x\mathcal{A}=b\mathcal{A}$ and $\mathcal{A}x=\mathcal{A}c$. Thus $x=a^{(b,c)}$. Therefore $a\in \mathcal{A}_{b,c}^{\tiny\textcircled{W}}$.

$(1)\Leftrightarrow (3)\Leftrightarrow (4)\Leftrightarrow (5)$ These are proved by Theorem 2.1, Theorem 3.3 and Lemma 3.4.\end{proof}

\begin{cor} Let $A,B,C\in {\Bbb C}^{n\times n}$, with conjugate transpose $*$ as the involution. Then the following are equivalent:\end{cor}
\begin{enumerate}
\item [(1)] $A$ has weak $(B,C)$-EP inverse.
\item [(2)] There exist $X,Y\in {\Bbb C}^{n\times n}$ such that $A=X+Y, X^*Y=YX=0, X~\mbox{has} ~(B,C)-EP ~\mbox{inverse}, Y~\mbox{is nilpotent}.$
\item [(3)] There exists an idempotent $E\in {\Bbb C}^{n\times n}$ such that $A+E$ is invertible, $(A^D)^*AE=0, (I_n-E)B=B, C(I_n-E)=C, AE=EAE$ is nilpotent.
\end{enumerate}
\begin{proof} As is well known, every complex has Drazin inverse. This completes the proof by Theorem 3.5.\end{proof}

We illustrate Corollary 3.6 be the following numerical example.

\begin{exam}\end{exam} Let $A=\left(
\begin{array}{cccc}
1&1&0&0\\
0&1&0&0\\
0&0&0&1\\
0&0&0&0
\end{array}
\right)\in {\Bbb C}^{4\times 4}$. Then $A^{\tiny\textcircled{W}}=\left(
\begin{array}{cccc}
1&-1&0&0\\
0&1&0&0\\
0&0&0&0\\
0&0&0&0
\end{array}
\right).$ Set $B=\left(
\begin{array}{cccc}
1&1&1&-1\\
0&1&-1&1\\
0&0&0&0\\
0&0&0&0
\end{array}
\right), C=\left(
\begin{array}{cccc}
1&1&0&0\\
1&-1&0&0\\
1&1&0&0\\
1&1&0&0
\end{array}
\right)$. Choose $E=I-AA^{\tiny\textcircled{W}}$. Then $E^2=E=\left(
\begin{array}{cccc}
0&0&0&0\\
0&0&0&0\\
0&0&1&0\\
0&0&0&1
\end{array}
\right).$ We verify that $A+E~\mbox{is invertible}, ~(A^*A^2E)^*=A^*A^2E, (I_n-E)B=B, C(I_n-E)=C, AE=EAE$ is nilpotent.
By virtue of Corollary 5.3, $A$ has $(B,C)$-EP inverse, i.e., $A^{\tiny\textcircled{W}}=A^{(B,C)}$.\\

\section{relations with associated generalized inverses}

This section is devoted to investigating the generalized $(b,c)$-EP properties through related generalized inverses. We proceed to derive

\begin{thm} Let $a\in \mathcal{A}$. Then $a\in \mathcal{A}_{b,c}^{\tiny\textcircled{g}}$ if and only if $a\in \mathcal{A}^{\tiny\textcircled{g}}$ and $cab\in \mathcal{A}^{-}$ such that $$a^{\tiny\textcircled{g}}=b(cab)^{-}c, b(cab)^{-}(cab)=b, (cab)(cab)^{-}c=c.$$\end{thm}
\begin{proof} $\Longrightarrow $ Obviously, $a\in \mathcal{A}^{\tiny\textcircled{g}}$. By hypothesis, $a^{\tiny\textcircled{g}}=a^{(b,c)}.$
Then $a^{\tiny\textcircled{g}}ab=b$ and $caa^{\tiny\textcircled{g}}=c$. In view of~\cite[Proposition 3.3]{WCC}, $cab\in \mathcal{A}^{-}$.
Since $a^{\tiny\textcircled{g}}\in b\mathcal{A}c$, we can find some elements
$x,y\in \mathcal{A}$ such that $a^{\tiny\textcircled{g}}=bx=yc$. Then $b=(yc)ab=y(cab)=(ycab)(cab)^{-}cab=b(cab)^{-}cab.$
Moreover, we have $c=cabx=(cab)(cab)^{-}(cab)x=(cab)(cab)^{-}c$.

We easily check that
\begin{align}
b(cab)^{-}cab(cab)^{-}c &= b(cab)^{-}c, \\
b(cab)^{-}cab &= b,\\
cab(cab)^{-}c &= c.
\end{align}
Therefore $b(cab)^{-}c=a^{(b,c)}$, and then $a^{\tiny\textcircled{g}}=b(cab)^{-}c,$ as desired.

$\Longleftarrow $ Since $a^{\tiny\textcircled{g}}=b(cab)^{-}c, b(cab)^{-}(cab)=b, cab(cab)^{-}c=c,$ we check that
the preceding identities $(4.1)-(4.3)$ hold. Then $b(cab)^{-}c=a^{(b,c)}$. Accordingly, $a^{\tiny\textcircled{g}}=a^{(b,c)},$ as required.\end{proof}

\begin{cor} Let $a\in \mathcal{A}$. Then $a\in \mathcal{A}_{b,c}^{\tiny\textcircled{g}}$ if and only if $a\in \mathcal{A}^{\tiny\textcircled{g}}$ and $$\begin{array}{c}
a^{\tiny\textcircled{g}}a\in b\mathcal{A}ca, aa^{\tiny\textcircled{g}}\in ab\mathcal{A}c,\\
(1-a^{\tiny\textcircled{g}}a)b=0, c(1-aa^{\tiny\textcircled{g}})=0;\\
ca(1-a^{\tiny\textcircled{g}}a)=0, (1-aa^{\tiny\textcircled{g}})ab=0.
\end{array}$$\end{cor}
\begin{proof} $\Longrightarrow $ In view of Theorem 4.1 and ~\cite[Corollary 2.3]{C7}, we easily check that
$a^{\tiny\textcircled{g}}a=a^{\tiny\textcircled{g}}aa^{\tiny\textcircled{g}}a\in b\mathcal{A}ca$.
Likewise, we have $aa^{\tiny\textcircled{g}}\in ab\mathcal{A}c.$
Moreover, we verify that
$$\begin{array}{rll}
(1-a^{\tiny\textcircled{g}}a)b&=&b-a^{\tiny\textcircled{g}}(aa^{\tiny\textcircled{g}}a)b=0,\\
c(1-aa^{\tiny\textcircled{g}})&=&c-c(aa^{\tiny\textcircled{g}}a)a^{\tiny\textcircled{g}}=0;\\
ca(1-a^{\tiny\textcircled{g}}a)&=&[c(aa^{\tiny\textcircled{g}}a)a^{\tiny\textcircled{g}}]a(1-a^{\tiny\textcircled{g}}a)=0, \\
(1-aa^{\tiny\textcircled{g}})ab&=&(1-aa^{\tiny\textcircled{g}})a[a^{\tiny\textcircled{g}}(aa^{\tiny\textcircled{g}}a)b]=0,
\end{array}$$ as required.

$\Longleftarrow $ In view of ~\cite[Corollary 2.3]{C7}, we have $a^{\tiny\textcircled{g}}=a^{\tiny\textcircled{g}}(aa^{\tiny\textcircled{g}}a)a^{\tiny\textcircled{g}}$.
By hypothesis, we verify that
$$\begin{array}{rll}
a^{\tiny\textcircled{g}}&=&a^{\tiny\textcircled{g}}aa^{\tiny\textcircled{g}}\in b\mathcal{A}\bigcap \mathcal{A}c,\\
a^{\tiny\textcircled{g}}(aa^{\tiny\textcircled{g}}a)b&=&a^{\tiny\textcircled{g}}ab=b,\\
c(aa^{\tiny\textcircled{g}}a)a^{\tiny\textcircled{g}}&=&caa^{\tiny\textcircled{g}}=c.
\end{array}$$ Therefore $a^{\tiny\textcircled{g}}=(aa^{\tiny\textcircled{g}}a)^{(b,c)}$, as desired.\end{proof}

Let $X$ be a complex Banach space, and let $\mathcal{B}(X)$ denote the Banach algebra of all bounded linear operators on $X$. For any operator $A\in
\mathcal{B}(X)$, we denote its null space by $\mathcal{N}(A)=\{ x\in X | Ax = 0\}$ and its range by $\mathcal{R}(A) = \{ Ax | x\in X\}$.

\begin{cor} Let $A\in \mathcal{B}(X)$. Then $A\in \big(\mathcal{B}(X)\big)^{{\tiny\textcircled{g}}_{(B,C)}}$ if and only if $A\in \big(\mathcal{B}(X)\big)^{\tiny\textcircled{g}}$ and $$
\begin{array}{rll}
\mathcal{R}(A^{\tiny\textcircled{g}}A)&=&\mathcal{R}(B),
\mathcal{N}(A^{\tiny\textcircled{g}}A)=\mathcal{N}(CA),\\
\mathcal{R}(AA^{\tiny\textcircled{g}})&=&\mathcal{R}(AB),
\mathcal{N}(AA^{\tiny\textcircled{g}})=\mathcal{N}(C).
\end{array}$$ \end{cor}
\begin{proof} $\Longrightarrow $ In view of Corollary 4.2, $A\in \big(\mathcal{B}(X)\big)^{\tiny\textcircled{g}}$ and $$\begin{array}{c}
A^{\tiny\textcircled{g}}A\in B\big(\mathcal{B}(X)\big)CA, AA^{\tiny\textcircled{g}}\in AB\big(\mathcal{B}(X)\big)C,\\
(I-A^{\tiny\textcircled{g}}A)B=0, C(I-AA^{\tiny\textcircled{g}})=0;\\
CA(I-A^{\tiny\textcircled{g}}A)=0, (I-AA^{\tiny\textcircled{g}})AB=0.
\end{array}$$ Hence, $$\begin{array}{rll}
\mathcal{R}(A^{\tiny\textcircled{g}}A)&=&\mathcal{R}(B),\\
\mathcal{R}(AA^{\tiny\textcircled{g}})&=&\mathcal{R}(AB),
\end{array}$$

Obviously, $\mathcal{N}(A^{\tiny\textcircled{g}}A)\subseteq \mathcal{N}(CA)$. Since $CA=CAA^{\tiny\textcircled{g}}A$, we see that
$\mathcal{N}(CA)\subseteq \mathcal{N}(A^{\tiny\textcircled{g}}A).$ This implies that
$$\mathcal{N}(A^{\tiny\textcircled{g}}A)=\mathcal{N}(CA).$$

Clearly, we have $\mathcal{N}(AA^{\tiny\textcircled{g}})\subseteq \mathcal{N}(C)$. As $C=CAA^{\tiny\textcircled{g}}$, we deduce that
$\mathcal{N}(C)\subseteq \mathcal{N}(AA^{\tiny\textcircled{g}})$. Therefore $\mathcal{N}(AA^{\tiny\textcircled{g}})=\mathcal{N}(C)$, as required.

$\Longleftarrow $ By hypothesis, $A\in \big(\mathcal{B}(X)\big)^{\tiny\textcircled{g}}$.

Claim 1. $A^{\tiny\textcircled{g}}\in B\big(\mathcal{B}(X)\big)C$. Obviously, $A^{\tiny\textcircled{g}}=(A^{\tiny\textcircled{g}}A)A^{\tiny\textcircled{g}}\in
\mathcal{R}(B)$; hence, $A^{\tiny\textcircled{g}}\in B\big(\mathcal{B}(X)\big)$. On the other hand,
$\mathcal{N}(AA^{\tiny\textcircled{g}})=\mathcal{N}(C)$. Thus,
$\mathcal{R}((AA^{\tiny\textcircled{g}})^*)=\mathcal{R}(C^*)$. Hence, $\mathcal{R}(AA^{\tiny\textcircled{g}})=\mathcal{R}(C^*)$.
Thus $$(A^{\tiny\textcircled{g}})^*=(AA^{\tiny\textcircled{g}})(A^{\tiny\textcircled{g}})^*\in C^*\big(\mathcal{B}(X)\big).$$
Accordingly, $A^{\tiny\textcircled{g}}\in \big(\mathcal{B}(X)\big)C$. Therefore $$A^{\tiny\textcircled{g}}=A^{\tiny\textcircled{g}}AA^{\tiny\textcircled{g}}\in B\big(\mathcal{B}(X)\big)C.$$

Claim 2. $A^{\tiny\textcircled{g}}AB=B.$  Write $B=A^{\tiny\textcircled{g}}AY$ for some $Y\in \mathcal{B}(X)$. Then $A^{\tiny\textcircled{g}}AB=
(A^{\tiny\textcircled{g}}A)(A^{\tiny\textcircled{g}}AY)=A^{\tiny\textcircled{g}}AY=B$, as required.

Claim 3. $CAA^{\tiny\textcircled{g}}=C$. Since $I-AA^{\tiny\textcircled{g}}\in \mathcal{N}(AA^{\tiny\textcircled{g}})$, we deduce that
$I-AA^{\tiny\textcircled{g}}\in \mathcal{N}(C)$. Then $CAA^{\tiny\textcircled{g}}=C$.\end{proof}

Next, we turn our attention to the relationships between generalized group inverses and generalized core-EP inverses.

\begin{thm} Let $a\in \mathcal{A}_{b,c}^{\tiny\textcircled{d}}$. Then $a\in \mathcal{A}_{(b,ca)}^{\tiny\textcircled{g}}$.
In this case, $$a_{(b,ca)}^{\tiny\textcircled{g}}=
\big(a_{b,c}^{\tiny\textcircled{d}}\big)^2a.$$\end{thm}
\begin{proof} By hypothesis, we have $$a_{b,c}^{\tiny\textcircled{d}}=a(a_{b,c}^{\tiny\textcircled{d}})^2, (aa_{b,c}^{\tiny\textcircled{d}})^*=aa_{b,c}^{\tiny\textcircled{d}}~\mbox{and} ~\lim\limits_{n\to \infty}||a^n-a_{b,c}^{\tiny\textcircled{d}}a^{n+1}||^{\frac{1}{n}}=0.$$ Set
$x=(a_{b,c}^{\tiny\textcircled{d}})^2a$. Then we check that
$$\begin{array}{rll}
ax^2&=&[a(a_{b,c}^{\tiny\textcircled{d}})^2][a(a_{b,c}^{\tiny\textcircled{d}})^2]a=(a_{b,c}^{\tiny\textcircled{d}})^2a=x,\\
(a^*a^2x)^*&=&(a^*a^2(a_{b,c}^{\tiny\textcircled{d}})^2a)^*=(a^*aa_{b,c}^{\tiny\textcircled{d}}a)^*\\
&=&a^*(aa_{b,c}^{\tiny\textcircled{d}})^*a=a^*(aa_{b,c}^{\tiny\textcircled{d}})a=a^*a^2x,\\
||a^n-xa^{n+1}||^{\frac{1}{n}}&=&||a^n-(a_{b,c}^{\tiny\textcircled{d}})^2a^{n+2}||^{\frac{1}{n}}\\
&=&||a^n-a_{b,c}^{\tiny\textcircled{d}}a^{n+1}+(a_{b,c}^{\tiny\textcircled{d}}(a^{n}-a_{b,c}^{\tiny\textcircled{d}}a^{n+1})a||^{\frac{1}{n}}\\
&=&||a^n-a_{b,c}^{\tiny\textcircled{d}}a^{n+1}||^{\frac{1}{n}}+||a_{b,c}^{\tiny\textcircled{d}}||^{\frac{1}{n}}||a^{n}-a_{b,c}^{\tiny\textcircled{d}}a^{n+1})||^{\frac{1}{n}}
||a||^{\frac{1}{n}}\\
&=&[1+||a_{b,c}^{\tiny\textcircled{d}}||^{\frac{1}{n}}||a||^{\frac{1}{n}}]||a^n-a_{b,c}^{\tiny\textcircled{d}}a^{n+1}||^{\frac{1}{n}}.
\end{array}$$ Hence $\lim\limits_{n\to \infty}||a^n-xa^{n+1}||^{\frac{1}{n}}=0.$ Therefore $a^{\tiny\textcircled{g}}=(a_{b,c}^{\tiny\textcircled{d}})^2a.$

We shall verify that $$a^{\tiny\textcircled{g}}=a^{(b,ca)}.$$ Clearly, $a^{\tiny\textcircled{g}}=a^{\tiny\textcircled{g}}aa^{\tiny\textcircled{g}}$.
By hypothesis, $a^{\tiny\textcircled{d}}=a^{(b,c)}$. Then we have $$\begin{array}{rll}
a^{\tiny\textcircled{d}}&\in &b\mathcal{A}\bigcap \mathcal{A}c,\\
b&=&a^{\tiny\textcircled{d}}ab,\\
c&=&caa^{\tiny\textcircled{d}}.
\end{array}$$ Hence, we verify that
$$\begin{array}{rll}
a^{\tiny\textcircled{g}}&=&(a_{b,c}^{\tiny\textcircled{d}})^2a\in b\mathcal{A}\bigcap \mathcal{A}ca,\\
a^{\tiny\textcircled{g}}ab&=&(a_{b,c}^{\tiny\textcircled{d}})^2a^2b\\
&=&(a_{b,c}^{\tiny\textcircled{d}})^2a^2a^{\tiny\textcircled{d}}ab\\
&=&a_{b,c}^{\tiny\textcircled{d}}aa^{\tiny\textcircled{d}}ab\\
&=&a^{\tiny\textcircled{d}}ab=b,\\
(ca)aa^{\tiny\textcircled{g}}&=&c[a^2(a_{b,c}^{\tiny\textcircled{d}})^2]a\\
&=&[caa_{b,c}^{\tiny\textcircled{d}}]a=ca.
\end{array}$$ Therefore $a^{\tiny\textcircled{g}}=a^{(b,ca)}.$\end{proof}

\begin{cor} Let $a\in \mathcal{A}_{b,c}^{\tiny\textcircled{d}}$. Then $a_{(b,ca)}^{\tiny\textcircled{g}}=x~\mbox{if and only if} ~ax^2=x, ax=a_{b,c}^{\tiny\textcircled{d}}a.$\end{cor}
\begin{proof} $\Longrightarrow $ In view of Theorem 4.4, $a\in \mathcal{A}_{(b,ca)}^{\tiny\textcircled{g}}$ and $x=a_{(b,ca)}^{\tiny\textcircled{g}}=\big(a_{b,c}^{\tiny\textcircled{d}}\big)^2a$.
Therefore $ax^2=x$ and $ax=a\big(a_{b,c}^{\tiny\textcircled{d}}\big)^2a=a_{b,c}^{\tiny\textcircled{d}}a$, as required.

$\Longleftarrow $ By hypotheses, $ax^2=x, ax=a_{b,c}^{\tiny\textcircled{d}}a$. Then we have
$x=ax^2=(ax)x=a_{b,c}^{\tiny\textcircled{d}}(ax)=\big(a_{b,c}^{\tiny\textcircled{d}}\big)^2a$. In light of Theorem 4.4,
$x=a_{(b,ca)}^{\tiny\textcircled{g}}$, as desired.\end{proof}

\begin{cor} Let $a\in \mathcal{A}^{\tiny\textcircled{d}}$. Then $a\in \mathcal{A}_{\big(a^d,(a^d)^*a\big)}^{\tiny\textcircled{g}}$.
In this case, $$a_{(a^d,(a^d)^*a)}^{\tiny\textcircled{g}}=(a^{\tiny\textcircled{d}})^2a.$$\end{cor}
\begin{proof} Since $a\in \mathcal{A}^{\tiny\textcircled{d}}$, it follows by ~\cite[Corollary 4.3]{C8} that
$a\in \mathcal{A}_{a^d,(a^d)^*}^{\tiny\textcircled{d}}$. According to Theorem 4.4, $a_{\big(a^d,(a^d)^*a\big)}^{\tiny\textcircled{g}}=(a^{\tiny\textcircled{d}})^2a,$ as required.\end{proof}

\begin{cor} Let $A\in {\Bbb C}^{n\times n}$. Then $A_{(A^D,(A^D)^*A)}^{\tiny\textcircled{W}}=(A^{\tiny\textcircled{D}})^2A$.\end{cor}
\begin{proof} Straightforward by Corollary 4.6.\end{proof}

Let $\mathcal{B}(X)$ be the Banach algebra of bounded linear operators over a Hilbert space $X$. Then the algebra $\mathcal{B}(X)$ is a Banach algebra with the adjoint operation as its proper involution. Let $A$ in $\mathcal{B}(X)^d$. In~\cite{MD4}, Mosic and D. Zheng introduced and studied weak group inverse for a Hilbert space operator. The weak group inverse of $A$ is defined by $A^{\bigotimes}=(A^{\tiny\textcircled{d}})^2A$.

\begin{cor} Let $A\in \mathcal{B}(X)^d$. Then $A_{(A^d,(A^d)^*A)}^{\bigotimes}=(A^{\tiny\textcircled{d}})^2A$.\end{cor}
\begin{proof} Since $A\in \mathcal{B}(X)^d$, it has generalized core inverse.
Obviously, $A^{\tiny\textcircled{d}}AA^d=A^d, (A^d)^*AA^{\tiny\textcircled{d}}=(A^d)^*$.
We claim that $A\in \mathcal{B}(X)_{(A^d, (A^d)^*)}^{\tiny\textcircled{d}}$. Therefore $A\in\mathcal{B}(X)_{(A^d,(A^d)^*A)}^{\bigotimes}$ by
Corollary 4.7.\end{proof}

To illustrate Corollary 4.8, we present the following numerical example.

\begin{exam}\end{exam} Let $A=\left(
\begin{array}{cc}
1&1\\
1&1
\end{array}
\right)\in {\Bbb C}^{2\times 2}$. Then $A^{\tiny\textcircled{\#}}=\left(
\begin{array}{cc}
\frac{1}{4}&\frac{1}{4}\\
\frac{1}{4}&\frac{1}{4}
\end{array}
\right).$ Let $S\in \mathcal{B}(\Bbb {C}\bigoplus {\Bbb C})$ be the linear operator given by the $2\times 2$ matrix $A$.

Let $X=\ell^2({\Bbb N})$ and the linear operator $S\in \mathcal{B}(X)$ is given by
$$Se_n=\frac{1}{2^n}e_{n+1} (n=0,1,2,\cdots ),$$ where $\{ e_n\}_{n=0}^{\infty}$ is a standard orthogonal basis, i.e., $S$ is defined by
$$S(x_0,x_1,x_2,\cdots )=(x_0, \frac{1}{2}x_1, \frac{1}{4}x_2, \frac{1}{8}x_3,\cdots ).$$
Then $||S^n||=1\cdot \frac{1}{2}\cdot \frac{1}{2^2}\cdot \frac{1}{2^{n-1}}=2^{-\frac{n(n-1)}{2}}$, and so $$\lim\limits_{n\to \infty}||S^n||^{\frac{1}{n}}=0.$$ Thus $S$ is quasinilpotent.

Set $H=\Bbb {C}\bigoplus {\Bbb C}\oplus X$. Construct a block diagonal operator $T=A\bigoplus S$, where $T|_{\Bbb {C}\bigoplus {\Bbb C}}=A\bigoplus 0_X$,
$T|_{X}=0_{2\times 2}\bigoplus S$. By ~\cite[Theorem 2.2]{C7}, $T\in \big(\mathcal{B}(H)\big)^{\tiny\textcircled{d}}$ and $$T^{\tiny\textcircled{d}}=diag(\left(
\begin{array}{cc}
\frac{1}{4}&\frac{1}{4}\\
\frac{1}{4}&\frac{1}{4}
\end{array}
\right),0_X).$$ Set $B=T^d=diag(\left(
\begin{array}{cc}
\frac{1}{4}&\frac{1}{4}\\
\frac{1}{4}&\frac{1}{4}
\end{array}
\right),0_X)$ and $C=(T^d)^*T=diag(\left(
\begin{array}{cc}
\frac{1}{2}&\frac{1}{2}\\
\frac{1}{2}&\frac{1}{2}
\end{array}
\right),0_X).$ In view of Corollary ???, we have
$$\begin{array}{rll}
T_{B,C}^{\bigotimes}&=&(T^{\tiny\textcircled{d}})^2T\\
&=&diag(\left(
\begin{array}{cc}
\frac{1}{4}&\frac{1}{4}\\
\frac{1}{4}&\frac{1}{4}
\end{array}
\right),0_X)^2diag(\left(
\begin{array}{cc}
1&1\\
1&1
\end{array}
\right),S)\\
&=&diag(\left(
\begin{array}{cc}
\frac{1}{4}&\frac{1}{4}\\
\frac{1}{4}&\frac{1}{4}
\end{array}
\right),0_X).
\end{array}$$

Finally, we reveal an intrinsic connection between the generalized $(b,b)$-EP inverse and the validity of the converse law for the generalized group inverse.¡±

\begin{thm} Let $a,b\in R$. Then the following are equivalent:\end{thm}
\begin{enumerate}
\item [(1)] $a\in R_{b,b}^{\tiny\textcircled{g}}$.
\item [(2)] $a\in R^{\tiny\textcircled{g}}, b\in R^{\#}$ and $aa^{\tiny\textcircled{g}}=bb^{\#}.$
\item [(3)] $a,ab\in R^{\tiny\textcircled{g}}, b\in R^{\#}, a^{\tiny\textcircled{g}}R=bR$ and $(ab)^{\tiny\textcircled{g}}=b^{\tiny\textcircled{g}}a^{\tiny\textcircled{g}}.$
\end{enumerate}
\begin{proof} $(1)\Rightarrow (2)$ By hypothesis, we have $a\in R^{\tiny\textcircled{g}}$ and
$(a^2a^{\tiny\textcircled{g}})^{\#}=(a^2a^{\tiny\textcircled{g}})^{(b,b)}$. In view of ~\cite[Theorem 5.1.18]{WC1}, $b\in R^{\#}$ and
$$(a^2a^{\tiny\textcircled{g}})(a^2a^{\tiny\textcircled{g}})^{\#}=bb^{\#}.$$ Therefore
$aa^{\tiny\textcircled{g}}=bb^{\#}.$

$(2)\Rightarrow (1)$ Clearly, $a^2a^{\tiny\textcircled{g}}\in \mathcal{A}^{\#}$. In view of ~\cite[Theorem 5.1.9]{WC1}, we have
$(a^2a^{\tiny\textcircled{g}})^{\#}=(a^2a^{\tiny\textcircled{g}})^{(b,c)}$. This implies that $a\in R_{b,b}^{\tiny\textcircled{g}}$, as asserted.

$(1)\Rightarrow (3)$ By the argument above, $a\in R^{\tiny\textcircled{g}}$ and $b\in R^{\#}$. In view of ~\cite[Proposition 6.1]{DM}, we have $a^{\tiny\textcircled{g}}R=bR$. Set $x=b^{\#}a^{\tiny\textcircled{g}}.$

Obviously, $b=a^{\tiny\textcircled{g}}ab$, and then $a^{\tiny\textcircled{g}}ab=b$.
By the argument above, we verify that
$$\begin{array}{rll}
abx&=&a(bb^{\#})a^{\tiny\textcircled{g}}=a(aa^{\tiny\textcircled{g}})a^{\tiny\textcircled{g}}\\
 &=&aa^{\tiny\textcircled{g}},\\
 abx^2&=&(aa^{\tiny\textcircled{g}})(b^{\#}a^{\tiny\textcircled{g}})\\
 &=&bb^{\#}(b^{\#}a^{\tiny\textcircled{g}})\\
 &=&b^{\#}a^{\tiny\textcircled{g}}=x,\\
 (ab)^*(ab)^2x&=&[(ab)^*(ab)][abx]\\
 &=&\big((ab)^*(ab)\big)\big(aa^{\tiny\textcircled{g}}\big)\\
 &=&\big((ab)^*(ab)\big)\big(bb^{\#}\big)=(ab)^*(ab),\\
 \big((ab)^*(ab)^2x\big)^*&=&(ab)^*(ab)^2x,\\
 1-x(ab)&=&1-(b^{\#}a^{\tiny\textcircled{g}})(ab)=1-bb^{\#},\\
 (1-xab)(ab)^n&=&(1-bb^{\#})(ab)^n=(ab)^n-(aa^{\tiny\textcircled{g}})(ab)^n\\
 &=&(ab)^n-a(a^{\tiny\textcircled{g}}ab)(ab)^{n-1}\\
 &=&(ab)^n-(ab)(ab)^{n-1}=0.
 \end{array}$$
 Therefore $(ab)^{\tiny\textcircled{g}}=x,$ as desired.

$(3)\Rightarrow (1)$ By hypothesis, $a,ab\in R^{\tiny\textcircled{g}}, b\in R^{\#}, a^{\tiny\textcircled{g}}R=bR$ and $(ab)^{\tiny\textcircled{g}}=b^{\#}a^{\tiny\textcircled{g}}.$ Write $a^{\tiny\textcircled{g}}=br$ and $b=a^{\tiny\textcircled{g}}s$ for some $r,s\in R$.
Then $$\begin{array}{rll}
a^{\tiny\textcircled{g}}&=&a(a^{\tiny\textcircled{g}})^2=br(a^{\tiny\textcircled{g}})^2\\
&=&bb^{\#}a^{\tiny\textcircled{g}}=b(ab)^{\tiny\textcircled{g}}\\
&=&bb^{\#}a^{\tiny\textcircled{g}}=bb^{\#}a^{\tiny\textcircled{g}}aa^{\tiny\textcircled{g}}\\
&=&bb^{\#}a^{\tiny\textcircled{g}}bb^{\#}\in bR\bigcap Rb.
\end{array}$$ Since $a^{\tiny\textcircled{g}}R=bR$, we see that $a^{\tiny\textcircled{g}}ab=b$.
Then $$aa^{\tiny\textcircled{g}}b=aa^{\tiny\textcircled{g}}[a^{\tiny\textcircled{g}}ab]=[a(a^{\tiny\textcircled{g}})^2]a=a^{\tiny\textcircled{g}}ab=b.$$
Then we verify that $$baa^{\tiny\textcircled{g}}=babb^{\#}a^{\tiny\textcircled{g}}bb^{\#}=b(aa^{\tiny\textcircled{g}}b)b^{\#}=b^2b^{\#}=b.$$ Therefore $a^{\tiny\textcircled{g}}=a^{(b,b)}$, as asserted.\end{proof}

\begin{cor} Let $a,b\in R$. Then the following are equivalent:\end{cor}
\begin{enumerate}
\item [(1)] $a\in R_{b,b}^{\tiny\textcircled{g}}$.
\item [(2)] $a,ab\in R^{\tiny\textcircled{g}}, b\in R^{\#}$, $a^{\tiny\textcircled{g}}=b(ab)^{\tiny\textcircled{g}}, b^{\#}=(ab)^{\tiny\textcircled{g}}a^2a^{\tiny\textcircled{g}}$.
\end{enumerate}
\begin{proof} $(1)\Rightarrow (2)$ In view of Theorem 4.10, we have $a^{\tiny\textcircled{g}}R=bR, aa^{\tiny\textcircled{g}}=bb^{\#}$ and $(ab)^{\tiny\textcircled{g}}=b^{\#}a^{\tiny\textcircled{g}}=b^{\#}a^{\tiny\textcircled{g}}bb^{\#}$.
Then $b(ab)^{\tiny\textcircled{g}}=bb^{\#}a^{\tiny\textcircled{g}}bb^{\#}=a^{\tiny\textcircled{g}}bb^{\#}=a^{\tiny\textcircled{g}}aa^{\tiny\textcircled{g}}  =a^{\tiny\textcircled{g}}$. Moreover, we see that
$(ab)^{\tiny\textcircled{g}}ab=b^{\#}a^{\tiny\textcircled{g}}ab=b^{\#}a^{\tiny\textcircled{g}}=b^{\#}b$. Hence,
$$b^{\#}=(ab)^{\tiny\textcircled{g}}abb^{\#}=(ab)^{\tiny\textcircled{g}}a^2a^{\tiny\textcircled{g}}.$$

$(2)\Rightarrow (1)$ By hypothesis, we have $a^{\tiny\textcircled{g}}=b(ab)^{\tiny\textcircled{g}}, b^{\#}=(ab)^{\tiny\textcircled{g}}a^2a^{\tiny\textcircled{g}}$. For any $n\geq 2$, we have
$$a^{\tiny\textcircled{g}}a^2a^{\tiny\textcircled{g}}=a^{\tiny\textcircled{g}}a^{n+1}(a^{\tiny\textcircled{g}})^{n-1},$$ and then
$$||aa^{\tiny\textcircled{g}}-a^{\tiny\textcircled{g}}a^2a^{\tiny\textcircled{g}}||^{\frac{1}{n}}\leq ||a^n-a^{\tiny\textcircled{g}}a^{n+1}||^{\frac{1}{n}}||a^{\tiny\textcircled{g}}||.$$
Since $\lim\limits_{n\to \infty}||a^n-a^{\tiny\textcircled{g}}a^{n+1}||^{\frac{1}{n}}=0,$ we deduce that
$\lim\limits_{n\to \infty}||aa^{\tiny\textcircled{g}}-a^{\tiny\textcircled{g}}a^2a^{\tiny\textcircled{g}}||^{\frac{1}{n}}=0.$ This implies that
$aa^{\tiny\textcircled{g}}=a^{\tiny\textcircled{g}}a^2a^{\tiny\textcircled{g}}$.
Thus we derive
$$\begin{array}{rll}
bb^{\#}&=&b[(ab)^{\tiny\textcircled{g}}a^2a^{\tiny\textcircled{g}}]\\
&=&[b(ab)^{\tiny\textcircled{g}}]a^2a^{\tiny\textcircled{g}}\\
&=&a^{\tiny\textcircled{g}}a^2a^{\tiny\textcircled{g}}\\
&=&aa^{\tiny\textcircled{g}}.
\end{array}$$ Therefore we complete the proof by Theorem 4.10.\end{proof}

\end{document}